\documentclass[11pt, a4paper]{amsart}

\usepackage[T1]{fontenc}
\usepackage{microtype}

\usepackage{amsmath}								%math
\usepackage{amssymb}
\usepackage{amsthm}
\usepackage{amscd}
\usepackage{amsfonts}
\usepackage{stmaryrd}

\usepackage{extarrows}
\usepackage[bookmarks=true, colorlinks=true, linkcolor=blue!50!black,
citecolor=orange!50!black, urlcolor=orange!50!black, pdfencoding=unicode]{hyperref}
\usepackage{color}

\usepackage{tikz}									
\usetikzlibrary{matrix}
\usetikzlibrary{patterns}
\usetikzlibrary{matrix}
\usetikzlibrary{positioning}
\usetikzlibrary{decorations.pathmorphing}
\usetikzlibrary{cd}

\usepackage{url}
\usetikzlibrary{babel}
\usepackage{tikz}
\usetikzlibrary{arrows.meta,positioning}
\usetikzlibrary {intersections}
\usetikzlibrary{calc}
\usetikzlibrary{decorations.pathreplacing,decorations.markings}
\tikzset{
  on each segment/.style={
    decorate,
    decoration={
      show path construction,
      moveto code={},
      lineto code={
        \path [#1]
        (\tikzinputsegmentfirst) -- (\tikzinputsegmentlast);
      },
      curveto code={
        \path [#1] (\tikzinputsegmentfirst)
        .. controls
        (\tikzinputsegmentsupporta) and (\tikzinputsegmentsupportb)
        ..
        (\tikzinputsegmentlast);
      },
      closepath code={
        \path [#1]
        (\tikzinputsegmentfirst) -- (\tikzinputsegmentlast);
      },
    },
  },
  mid arrow/.style={line width=1pt,postaction={decorate,decoration={
        markings,
        mark=at position .55 with -{\arrow[#1]{Stealth}}
      }}},
}

\usepackage{nicefrac}
\usepackage{fullpage}
\usepackage{enumerate}

\newtheorem*{theorem*}{Theorem}

\newtheorem{maintheorem}{Theorem}[section]
\newtheorem{maincorollary}[maintheorem]{Corollary}
									
\newtheorem{theorem}{Theorem}[section]
\newtheorem{lemma}[theorem]{Lemma}
\newtheorem{proposition}[theorem]{Proposition}

\theoremstyle{definition}
\newtheorem{definition}[theorem]{Definition}
\newtheorem{example}[theorem]{Example}

\newtheorem{remark}[theorem]{Remark}

\newtheoremstyle{myitemstyle}						% flexible theorem style
	{}			%Space above
	{}			%Space below
	{}			%Body font
	{}			%indent amount
	{}			%Thm head font
	{.}			%Punkte nach thm head
	{ }			%Abstand nach thm head
	{}			%Thm head spec
\theoremstyle{myitemstyle}
\newtheorem{myitemthm}{}

\newcommand{\hooklongrightarrow}{\lhook\joinrel\longrightarrow}

\newcommand{\R}{\mathbb{R}}

\newcommand{\Z}{\mathbb{Z}}

\newcommand{\C}{\mathbb{C}}

\DeclareMathOperator{\Pic}{Pic}

\DeclareMathOperator{\Hom}{Hom}

\DeclareMathOperator{\val}{val}

\DeclareMathOperator{\Div}{Div}

\DeclareMathOperator{\Jac}{Jac}
\DeclareMathOperator{\PDiv}{PDiv}

\DeclareMathOperator{\ord}{ord}

\let\Rat\relax
\DeclareMathOperator{\Rat}{Rat}

\let\div\relax
\DeclareMathOperator{\div}{div}

\newcommand{\X}{\mathfrak{X}}

\title{An Abel--Jacobi theorem for metrized complexes of Riemann surfaces} 

\date{}

\author{Maximilian C. E. Hofmann}
\address{Institut f\"ur Mathematik, Goethe--Universit\"at Frankfurt,
60325 Frankfurt am Main, Germany\\
Department of Mathematics, Stony Brook University, Stony Brook, NY 11794-3651, USA}
\email{maximilian.hofmann@stonybrook.edu}

\author{Martin Ulirsch}
\address{Institut f\"ur Mathematik, Goethe--Universit\"at Frankfurt,
60325 Frankfurt am Main, Germany}
\email{ulirsch@math.uni-frankfurt.de}

\begin{document}

\begin{abstract} Motivated by the recent surge of interest in the geometry of hybrid spaces, we prove an Abel--Jacobi theorem for a metrized complex of Riemann surfaces, generalizing both the classical Abel--Jacobi theorem and its tropical analogue. 
\end{abstract}

\maketitle

\setcounter{tocdepth}{1}
\tableofcontents

\section*{Introduction}

A recent trend at the intersection of Archimedean, non-Archimedean, and tropical geometry is the introduction of \emph{hybrid spaces}, which form a convenient framework to encode both geometric and combinatorial phenomena. Central highlights of this trend are the study of limits of volume forms on degenerations of Calabi-Yau varieties motivated by mirror symmetry \cite{BoucksomJonsson} and the construction of hybrid compactification of moduli spaces (see e.g.\ \cite{OdakaI, OdakaII} and \cite{AminiNicolussiI, AminiNicolussiII}). 

The simplest one-dimensional hybrid spaces are metrized complexes of algebraic curves, introduced by Amini and Baker \cite{AminiBaker} in order to provide us with a new framework to study limit linear series from both an algebraic and a tropical perspective (see \cite{BakerJensen} for a survey of the latter developments). Roughly speaking, a metrized complex $\mathfrak{X}$ of algebraic curves consists of a metric graph $\Gamma$ together with a smooth projective curve $X_v$ at every vertex $v$ of $\Gamma$ that can be thought of as a normalization of a component of the special fiber in a semistable degeneration of a smooth projective curve (see Definition \ref{def_mCRS} below for a precise definition and Figure \ref{figure_mCRS} for some geometric intuition). In their work, Amini and Baker in particular proved a Riemann--Roch formula that generalizes both the classical Riemann--Roch formula on smooth projective algebraic curves and its tropical analogue \cite{BakerNorine, MikhalkinZharkov, GathmannKerber, AminiCaporaso}. 

\begin{figure}[ht]
\centering
\begin{tikzpicture}
    %\pic{torus={1cm}{2.8mm}{70}}
    \draw (0,-0.1) ellipse (1.5 and 0.7);
    \draw [name path=test1] (-0.5,0) .. controls (-0.5,-0.25) and (0.5,-0.25) .. (0.5,0);
    \draw (-0.5,0) .. controls (-0.5,-0.25) and (0.5,-0.25) .. (0.5,0);
    \path [name path=test2] (-0.5,-0.1) -- (0.5,-0.1);
    \draw [name intersections={of=test1 and test2}]
    (intersection-1) .. controls (intersection-1 |- 0,0.1) and (intersection-2 |- 0,0.1) .. (intersection-2); 

    \begin{scope}[xshift=2cm, yshift=-2cm]
        \draw (0,-0.1) ellipse (1.5 and 0.7);
    \draw [name path=test1] (-0.5,0) .. controls (-0.5,-0.25) and (0.5,-0.25) .. (0.5,0);
    \draw (-0.5,0) .. controls (-0.5,-0.25) and (0.5,-0.25) .. (0.5,0);
    \path [name path=test2] (-0.5,-0.1) -- (0.5,-0.1);
    \draw [name intersections={of=test1 and test2}]
    (intersection-1) .. controls (intersection-1 |- 0,0.1) and (intersection-2 |- 0,0.1) .. (intersection-2); 
    \end{scope}

    \begin{scope}[xshift=4cm]
        \draw (0,0) circle (1);
        \draw (-1,0) .. controls (-1,-0.3) and (1,-0.3) .. (1,0);
        \draw [dotted] (-1,0) .. controls (-1,0.3) and (1,0.3) .. (1,0);
    \end{scope}

    \begin{scope}[xshift=7cm]
        \draw (0,0) circle (1);
        \draw (-1,0) .. controls (-1,-0.3) and (1,-0.3) .. (1,0);
        \draw [dotted] (-1,0) .. controls (-1,0.3) and (1,0.3) .. (1,0);
    \end{scope}

    \path % let's define some points:
    let
      \p1 = (1,0),
      \p2 = (3.8,0.5),
      \p3 = (0.7,-0.3),
    in
      coordinate (p1) at (\p1)
      coordinate (p2) at (\p2)
      coordinate (p3) at (\p3)
      coordinate (p4) at (1,-2)
      coordinate (p5) at (3,-2)
      coordinate (p6) at (3.8,-0.7)
      coordinate (p7) at (4.6,0.1)
      coordinate (p8) at (6.5,0.35);

  \draw (p1) -- (p2);
  \draw (p3) -- (p4);
  \draw (p5) -- (p6);
  \draw (p7) -- (p8);
  \fill[black] (p1) circle [radius=2pt];
  \fill[black] (p2) circle [radius=2pt];
  \fill[black] (p3) circle [radius=2pt];
  \fill[black] (p4) circle [radius=2pt];
  \fill[black] (p5) circle [radius=2pt];
  \fill[black] (p6) circle [radius=2pt];
  \fill[black] (p7) circle [radius=2pt];
  \fill[black] (p8) circle [radius=2pt];
     
\end{tikzpicture}
\caption{Geometric realization of a metrized complex of Riemann surfaces of genus $3$.}\label{figure_mCRS}
\end{figure}
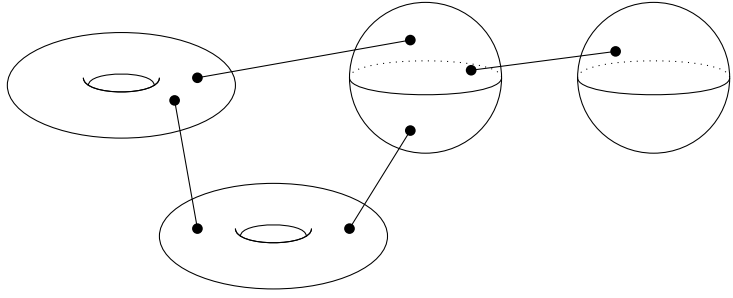

Another classical result from the geometry of compact Riemann surfaces, the Abel--Jacobi Theorem, also admits a tropical analogue \cite{MikhalkinZharkov, BakerFaber}. But, so far, a version of the Abel--Jacobi theorem for metrized complexes of smooth projective curves is missing from the literature. The central goal of this note is to change this state of affairs by proving an Abel--Jacobi theorem for metrized complexes of smooth projective curves over $\mathbb{C}$ or, for short, metrized complexes of compact Riemann surfaces (mCRS). 

\subsection*{Main result} 
In this article, we first recall the necessary background from the theory of divisors on metric graphs in Section \ref{section_metricgraphs} and on metrized complexes of Riemann surfaces in Section \ref{section_mCRS} (originally developed in \cite{AminiBaker}).
In particular, we write $\Div$, $\Div_0$, $\Pic$, and $\Pic^0$, for the groups of divisors, divisors of degree zero, divisor classes, and divisor classes of degree zero respectively on a compact Riemann surface, a metric graph, and a metrized complex of Riemann surfaces. 

In Section \ref{section_JacobianmCRS} below we then define the \emph{Jacobian} of a metrized complex of Riemann surfaces $\X$ as a real torus that interpolates between both the classical Jacobian $\Jac(X_v)$ of the Riemann surfaces $X_v$ at the vertices and the Jacobian $\Jac(\Gamma)$ of the underlying metric graph $\Gamma$, which is obtained from $\X$ by collapsing each $X_v$ to a point. We go on to construct a suitable Abel-Jacobi map from $\X$ to $\Jac(\X)$ which can be extended to a map $A_{\X,0}$ from the degree $0$ divisors $\Div_0(\X)$ on $\X$ to $\Jac(\X)$ and which generalizes the Abel-Jacobi maps $A_{v,0}\colon \Div_0(X_v)\rightarrow \Jac(X_v)$ of the Riemann surfaces $X_v$ at the vertices of $\Gamma$ and the Abel-Jacobi map $A_{\Gamma,0}\colon \Div_0(\Gamma)\rightarrow \Jac(\Gamma)$ of the underlying metric graph $\Gamma$.

\begin{maintheorem}[Theorem \ref{thm:abel_jacobi_for_mCRS}]\label{mainthm_AbelJacobimCRS} Let $\mathfrak{X}$ be a metrized complex of Riemann surfaces with underlying metric graph $\Gamma$ and with compact Riemann surfaces $X_v$ at the vertices $v$ of $\Gamma$.
    There is a canonical isomorphism of short exact sequences given by the following commutative diagram
    \[\begin{tikzcd}
    	0 & {\bigoplus_{v \in V}\mathrm{Pic}^0(X_v)} & {\mathrm{Pic}^0(\X)} & {\mathrm{Pic}^0(\Gamma)} & 0 \\
    	0 & {\bigoplus_{v \in V} \Jac(X_v)} & {\Jac(\X)} & {\Jac(\Gamma)} & 0
    	\arrow[from=1-1, to=1-2]
    	\arrow[from=1-2, to=1-3]
    	\arrow["A_{X,0} = \bigoplus_{v \in V} A_{v,0}",from=1-2, to=2-2]
    	\arrow[from=1-3, to=1-4]
    	\arrow["{A_{\X,0}}", from=1-3, to=2-3]
    	%\arrow["u"', curve={height=12pt}, dashed, from=1-4, to=1-3]
    	\arrow[from=1-4, to=1-5]
    	\arrow["{A_{\Gamma,0}}", from=1-4, to=2-4]
    	\arrow[from=2-1, to=2-2]
    	\arrow[from=2-2, to=2-3]
    	\arrow[from=2-3, to=2-4]
    	\arrow[from=2-4, to=2-5]
    \end{tikzcd},\]
    where all vertical maps are given by the respective Abel-Jacobi maps in degree zero. Furthermore, both rows are right-split, and each choice for a set of basepoints, where we choose one basepoint $p_v \in X_v$ for each Riemann surface $X_v$, gives rise to such a splitting.
\end{maintheorem}

Theorem \ref{mainthm_AbelJacobimCRS} implies the following Corollary \ref{maincor_AbelJacobi}, which one might want to call an \emph{Abel-Jacobi-Theorem} for metrized complex of Riemann surfaces. 

\begin{maincorollary}\label{maincor_AbelJacobi}
Let $\X$ be a mCRS and let $D \in \Div_0(\X)$ be a divisor on $\X$ of degree $0$. Then, the Abel--Jacobi map $A_{\X,0}$ is surjective, and $D$ is a principal divisor if and only if $A_{\X,0}(D) = 0$ in $\Jac(\X)$. In other words, the Abel--Jacobi map induces a canonical isomorphism
$$
\Jac(\X) \cong \frac{\Div_0(\X)}{\mathrm{PDiv}(\X)} \ .
$$
\end{maincorollary}

We point out that this result reduces to the Abel-Jacobi theorem for Riemann surfaces, when the underlying metric graph is a point (see Example \ref{ex:Riemann_surfaces_as_mCRS} below), and to the Abel-Jacobi-Theorem for metric graphs, when the Riemann surfaces at the vertices are all projective lines (see Example \ref{ex:metric_graphs_as_mCRS} below). We do not obtain a new proof of either theorem since our proof uses both of them.

\subsection*{The analogy with logarithmic geometry} 
Logarithmic geometry, in the sense of Fontaine--Kato--Illusie, may be viewed as another take on the principles of hybrid geometry, as it encodes both classical algebro-geometric and tropical data in a convenient framework. In \cite{FosterRanganathanTalpoUlirsch}, it has been observed that logarithmic curves are closely related to metrized complexes of curves and, in particular, the geometry of the logarithmic Picard group, which has recently been brought to perfection in \cite{MolchoWise}, is closely related to the geometry of the Picard group of an mCRS. In particular, one should view our Theorem \ref{mainthm_AbelJacobimCRS} in analogy with the quotient presentation of the logarithmic Picard group constructed in \cite[Corollary 4.6.3]{MolchoWise}.

\subsection*{Acknowledgements}
The authors would like to thank Andreas Gross for helpful discussions en route to this article. Remarks from the anonymous referee have significantly improved the structure of our main argument and, for example, also lead to a simpler proof of Proposition \ref{prop_seshomology}. We thank the referee for generously sharing their ideas.

\subsection*{Funding} This project has received funding from the Deutsche Forschungsgemeinschaft (DFG, German Research Foundation) TRR 326 \emph{Geometry and Arithmetic of Uniformized Structures}, project number 444845124, from the DFG Sachbeihilfe \emph{From Riemann surfaces to tropical curves (and back again)}, project number 456557832, as well as the DFG Sachbeihilfe \emph{Rethinking tropical linear algebra: Buildings, bimatroids, and applications}, project number 539867663, within the  
SPP 2458 \emph{Combinatorial Synergies}.

\section{Divisors on metric graphs}\label{section_metricgraphs}
In this section, we recall the basic theory of divisors and their  linear equivalence on metric graphs up to the Abel--Jacobi theorem. General references include the surveys  \cite{BakerJensen} and \cite{Len} as well as \cite{MikhalkinZharkov} and \cite{BakerFaber}.

There are many essentially equivalent ways to define a metric graph. In this work, we choose the following definition, which we adapted from \cite{Mugnolo}.

\begin{definition}[Metric Graphs]
    Let $G = (V,E)$ be a connected multigraph with finite vertex set $V$ and edge set $E$ (possibly with loops) together with a length function $\ell : E \to \R_{>0}$. Choose an orientation for the edges and write $e^+$ for the starting vertex and $e^-$ for the end vertex of the edge $e$ under this orientation. For every edge $e$ with end points/vertices $e^+,e^-$, let $I_e \subseteq \R$ be an interval of length $\ell(e)$ with endpoints $x_{e^+}^e = \min I_e$ and $x_{e^-}^e = \mathrm{max}\, I_e$. If $e$ is a loop, we can immediately identify the endpoints of $I_e$ with each other and simply denote the resulting point by $x_{v,e}$. A \textit{metric graph} $\Gamma$ with \textit{model} $G$ is a geometric realization of $G$. We make this precise by considering the equivalence relation $\sim$ on the disjoint union of the intervals $I_e$ generated by
    $$
    x_{v}^e \sim x_{v'}^e \Longleftrightarrow v = v'.
    $$
    Then, we define
    $$
    \Gamma := \bigsqcup_{e \in E} I_e \Big/\sim.
    $$
    Note that equipping $\Gamma$ with the \textit{shortest-path}-metric in the obvious way makes $\Gamma$ into a metric space (whose structure does not depend on the orientation chosen in the beginning). The \emph{valency} of a point $p \in \Gamma$ is the number of endpoints of intervals $I_e$ glued together at $p$ or $2$ if $p$ is in the interior of one of the $I_e$.
\end{definition}

In this definition, metric graphs are always compact. The reader should also note that two different models can lead to isometric metric graphs. However, if $G$ and $G'$ are two such models, they always admit a common refinement obtained by successively replacing an edge by two new edges connected via a vertex of degree $2$. The length of the old edge equals the sum of the lengths of the two new edges.

Given a metric graph $\Gamma$, we can define its \emph{Euler characteristic} (respectively its \emph{genus}) to be the Euler characteristic (genus) of any model of $\Gamma$. It is easy to see that these numbers are independent of the  chosen model.

A continuous function $f : \Gamma \to \R$ on a metric graph $\Gamma$ is called \textit{rational} if it is piecewise linear with integer slopes. The abelian group of rational functions will be denoted by $\Rat(\Gamma)$. Now, let $0 \neq f \in \Rat(\Gamma)$ be rational and $p \in \Gamma$. The order $\ord_p(f)$ of $f$ in $p$ is defined to be the sum of the outgoing slopes of $f$ at $p$. Note that the order of a rational functions is nonzero only at finitely many points. A rational function on $\Gamma$ should be thought of as the tropical analogue of a meromorphic function on a Riemann surface.

A \textit{divisor} $D = \sum_{p \in \Gamma} D(p) \cdot p$ is an element of the free abelian group $\Div(\Gamma)$ generated by the points $p \in \Gamma$. A divisor is called \textit{effective} if $D(p) \geq 0$ for all $p \in \Gamma$. The degree of $D$ is 
$$
\deg(D) = \sum_{p \in \Gamma} D(p),
$$
and the set of all degree $d$ divisors is denoted $\Div_d(\Gamma)$. To a rational function $f \in \Rat(\Gamma)$, we associate 
$$
\div(f) = \sum_{p \in \Gamma} \ord_p(f) \cdot p.
$$
Divisors of this form are called \textit{principal}, and we denote the set of all principal divisors by $\mathrm{PDiv}(\Gamma)$. Two divisors $D_1, D_2$ are called linearly equivalent, written $D_1 \sim D_2$, if their difference is principal. 

To formulate an Abel-Jacobi theorem for metric graphs, we need a suitable definition that is analogous to the notion of holomorphic $1$-forms on Riemann surfaces.

\begin{definition}[Harmonic $1$-Forms]
    Let $G = (V,E)$ be a connected multigraph. Again, we also choose an orientation for the edges of $G$ and write $e^+$ for the starting vertex and $e^-$ for the end vertex of an edge $e \in E$. A $1$-form on $G$ is an element
    $$
    \omega = \sum_{e \in E} \omega_e de
    $$
    of the $\R$-vector space with formal basis $\{de \vert e \in E\}$. A $1$-form is called \textit{harmonic} if
    $$
    \sum_{\substack{e \in E \\ e^+ = v}} \omega_e = \sum_{\substack{e \in E \\ e^- = v}} \omega_e
    $$
    for all $v \in V$. We denote the subspace of harmonic $1$-forms by $\Omega(G)$.
    Now, let $G$ and $G'$ be two oriented models for a metric graph $\Gamma$. By finding a common refinement of $G$ and $G'$, it is easy to see that there is a canonical isomorphism
    $\Omega(G) \cong \Omega(G')$.
    We can therefore define
    $$
    \Omega(\Gamma) := \Omega(G)
    $$
    for an oriented model $G$ of $\Gamma$. We call elements of this space \textit{harmonic} $1$\textit{-forms} on $\Gamma$.
\end{definition}

In analogy to the construction of the Jacobian of a Riemann surface, we now wish to embed $H_1(\Gamma,\Z)$ into $\Omega(\Gamma)^* := \Hom_\R(\Omega(\Gamma),\R)$. Let 
$$
\Gamma = \bigsqcup_{e \in E} I_e \Big/ \sim
$$
be a metric graph with oriented model $G = (V,E)$. We call a path $\gamma : [0,1] \to \Gamma$ \textit{piecewise smooth} if its restriction to the preimages of the $I_e$ are piecewise smooth. Note that the $I_e$ are $1$-dimensional manifolds with boundary and an orientation induced by the oriented model $G$. We can now define the path integral of $\omega = \sum_{e \in E} \omega_e de \in \Omega(\Gamma)$ via pullback:
$$
\int_\gamma \omega := \sum_{e \in E} \omega_e \int (\gamma|_{\gamma^{-1}(I_e)})^* de.
$$
This integral only depends on the homotopy class of a path, and so we obtain our desired well defined injection

\begin{equation*}\begin{split}
H_1(\Gamma) &\hooklongrightarrow \Omega(\Gamma)^*\\
[\gamma] &\longmapsto \int_\gamma := (\omega \mapsto \int_\gamma \omega),
\end{split}\end{equation*}
where $\gamma$ is a smooth representative of $[\gamma]$. We call the quotient
$$
\Jac(\Gamma) := \Omega(\Gamma)^*/H_1(\Gamma)
$$
the \textit{Jacobian} of $\Gamma$.

Choosing a basepoint $p \in \Gamma$, there is a well-defined \textit{Abel-Jacobi} map
\begin{equation*}\begin{split}
A_\Gamma : \Gamma &\longrightarrow \Jac(\Gamma)\\
p &\longmapsto \int_{p_0}^p := \left[\int_\gamma\right]
\end{split}\end{equation*}
where $\gamma$ is some piecewise smooth path from $p_0$ to $p$. We can extend $A_\Gamma$ linearly to the set of all divisors $\Div(\Gamma)$. If we restrict ourselves to divisors of degree $0$, $A_\Gamma$ does no longer depend on the chosen basepoint. We write $A_{\Gamma,0}$ for this map, which is also called \emph{Abel-Jacobi map}.

We now have:

\begin{theorem}[Abel-Jacobi theorem for metric graphs]\label{thm:abel_jacobi_theorem_for_metric_graphs}
    Let $D \in \Div_0(\Gamma)$ be a degree $0$ divisor. Then, $D$ is principal if and only if
    $$
    A_{\Gamma,0}(D) = 0.
    $$
    Furthermore, the map $A_{\Gamma,0}$ is surjective and therefore induces an isomorphism 
    $$
    \Jac(\Gamma) \cong \frac{\Div_0(\Gamma)}{\mathrm{PDiv}(\Gamma)}.
    $$
\end{theorem}

For a proof of Theorem \ref{thm:abel_jacobi_theorem_for_metric_graphs} and more details on the Abel-Jacobi theory of metric graphs, we refer the reader to \cite[Theorem 6.2]{MikhalkinZharkov} and \cite[Theorem 3.4]{BakerFaber}.

\section{Divisors on metrized complexes of Riemann surfaces}\label{section_mCRS} 
In this section, we recall the basic geometry of divisors on metrized complexes of Riemann surfaces, as introduced in \cite{AminiBaker}.

\begin{definition}\label{def_mCRS} A \textit{metrized complex of Riemann surfaces} (mCRS) is a tuple 
$$
\X = \big(G,\Gamma,\{X_v\}_{v\in V}, \{\mathcal{A}_v\}_{v \in V}\big)$$
consisting of 
\begin{itemize}
    \item a finite connected multigraph $G = (V,E)$ with vertex set $V \neq \emptyset$, edge set $E$ and a length function $\ell : E \to \R_{> 0}$ giving rise to a compact metric graph $\Gamma$ with model $G$,
    \item for all $v \in V$ a compact Riemann surface $X_v$, and
    \item for all $v \in V$ a finite set of distinct marked points $\mathcal{A}_v = \{x_v^e\}_{e \ni v} \subseteq X_v$ on $X_v$ together with a bijection between $\mathcal{A}_v$ and the edges $e \in E$ that are incident to $v$ (loops are counted twice).
\end{itemize}
 The \textit{geometric realization} $|\X|$ of $\X$ is defined to be the topological space given by taking the disjoint union of the Riemann surfaces $\bigsqcup_{v \in V} X_v$ and attaching the edges in $E$, thought of as intervals $I_e$ of length $\ell(e)$, to the marked points $x_v^e$ by gluing the endpoints of $I_e$ to $x_{v}^e$ and $x_{u}^e$, where $u$ and $v$ are the vertices incident to the edge $e$.
\end{definition}

By abuse of notation, we will frequently denote $|\X|$ by $\X$ as well. Furthermore, whenever needed, we will think of $V$ as being a subset of $\Gamma$ and of $E$ as being a subset of $\Gamma$ or $|\X|$. Under this identification, an abstract edge $e = \{v,u\} \in E$ shall always correspond to a closed subset $e$ of $\Gamma$ or $|\X|$, respectively. In particular, $e \cap \{v\} = \{v\} \subseteq \Gamma$, and $e \cap X_v = \{x_v^e\} \subseteq |\X|$, where we assume that $v \neq u$, i.e. $e$ is not a loop, to obtain the second equality.

\begin{example}[Riemann surfaces as special cases of mCRS]\label{ex:Riemann_surfaces_as_mCRS}
Every compact Riemann surface $X$ is canonically a mCRS 
$$
\X = \Big(\big(\{*\},\emptyset\big),\{*\},\{X\},\{\emptyset\}\Big)
$$
whose underlying (metric) graph is the one-point graph $\{\ast\}$.
\end{example}

The \textit{genus} $g(\X)$ of a mCRS $\X$ is defined to be the sum of the individual genera $g(X_v)$ of the Riemann surfaces $X_v$, for $v \in V$, and the genus $g(\Gamma)$ of $\Gamma$.

As for metric graphs and Riemann surfaces, a \textit{divisor}
$$
D = \sum_{p \in |\X|} D(p) \cdot p
$$
is an element of the free abelian group $\Div(\X)$ generated by the points of $|\X|$. In particular, $D(p) = 0$ for all but finitely many $p$. The degree of $D$ is $\deg(D) := \sum_{p \in |\X|} D(p)$. The $X_v$-\textit{part} $D_v \in \Div(X_v)$ of $D$ is given by
$$
D_v := \sum_{p \in X_v} D(p) \cdot p,
$$
and the $\Gamma$-\textit{part} $D_\Gamma \in \Div(\Gamma)$ is defined to be
$$
D_\Gamma := \sum_{p \in \Gamma \setminus V} D(p) \cdot p + \sum_{v \in V} \deg(D_v) \cdot v.
$$

A nonzero \textit{rational function} on $\X$ is given by a tuple $\mathfrak{f} = (\{f_v\}_{v \in V}, f_\Gamma)$ where $f_\Gamma$ is a rational function on $\Gamma$, and the $f_v$ are nonzero meromorphic functions on $X_v$. We call $f_\Gamma$ the $\Gamma$-part of $\mathfrak{f}$ and $f_v$ the $X_v$-part of $\mathfrak{f}$. The set of all rational functions is denoted by $\Rat(\X)$.

It is important to observe that this definition does not impose any compatibility conditions on the $X_v$-parts and the $\Gamma$-part of a rational function on $\X$.

As in the theory of Riemann surfaces or metric graphs, we can associate a divisor to a rational function on $\X$. To be precise, we define the divisor of the rational function 
$$
0 \neq \mathfrak{f} = (\{f_v\}_{v \in V}, f_\Gamma) \in \mathrm{Rat}(\X)
$$
by
$$
\div(\mathfrak{f}) := \sum_{p \in \X} \ord_p(\mathfrak{f}) \cdot p
$$
where
$$
\ord_p(\mathfrak{f}) := \left\{
\begin{array}{ll}
    \ord_p(f_v) & \textnormal{if } p \in X_v \setminus \mathcal{A}_v \\
    \mathrm{ord}_p(f_\Gamma) & \textnormal{if } p \in \Gamma \setminus V \\
    \ord_p(f_v) + \mathrm{slp}_e (f_\Gamma) & \textnormal{if } p = x_v^e \in \mathcal{A}_v.
\end{array}\right.
$$
Here, $\mathrm{slp}_e(f_\Gamma)$ denotes the outward pointing slope of $f_\Gamma$ at $v$ along $e$. A divisor $D \in \Div(\X)$ is called \textit{principal} if there exists a rational function $\mathfrak{f} \in \Rat(\X)$ such that $D = \div(\mathfrak{f})$, and two divisors are called linearly equivalent if their difference is principal. The set of principal divisors forms a subgroup of $\Div(\X)$, which we denote by $\mathrm{PDiv}(\X)$ .

If 
$$
0 \neq \mathfrak{f} = \big(\{f_v\}_{v \in V}, f_\Gamma\big) \in \Rat(\X)
$$
is a rational function on $\X$, we have $\deg(\mathfrak{f}) = \sum_{v \in V} \deg(\div(f_v)) + \deg(\div(f_\Gamma)) = 0$. Therefore, $\mathrm{PDiv}(\X) \subseteq \Div_0(\X)$. The reader should also note that while the equation
$$
\div(\mathfrak{f})_\Gamma = \div(f_\Gamma)
$$
always holds, we almost never have
$$
\div(\mathfrak{f})_v = \div(f_v).
$$
Hence, one cannot just apply the Abel-Jacobi theorem for Riemann surfaces to the $X_v$-part and the Abel-Jacobi theorem for metric graphs to the $\Gamma$-part of a divisor $D$ separately to decide whether $D$ is principal.

\begin{example}[Metric graphs as special cases of mCRS]\label{ex:metric_graphs_as_mCRS}
    Let $\X = \big(G,\Gamma,\{X_v\}_{v \in V},\{\mathcal{A}_v\}_{v \in V}\big)$ be a mCRS such that $g(X_v) = 0$ for each $v \in V$. Then, the divisor theories of $\X$ and $\Gamma$ are essentially equivalent since every pair of points on some $X_v$ is already linearly equivalent as the Riemann sphere has trivial Jacobian. In particular, two divisors $D,D' \in \Div(\X)$ are linearly equivalent if and only if $D_\Gamma$ and $D'_\Gamma$ are linearly equivalent. 
\end{example}

\section{The Jacobian of a metrized complex of Riemann surfaces}\label{section_JacobianmCRS}

In this section, we propose a natural definition for the Jacobian of a mCRS $\X$ that will make Theorem \ref{mainthm_AbelJacobimCRS} true. We think of $\X$ in terms of its geometric realization $\vert\X\vert$ and consider the group of singular $i$-chains on $\vert\X\vert$, which we denote by $C_i(\X)$. Write $d : C_i(\X) \to C_{i-1}(\X)$ for the usual boundary map. The ordinary homology groups of a topological space $X$ are denoted by $H_i(X)$, whereas $\widetilde{H}_i(X)$ denotes the $i$th reduced homology group of $X$.

\begin{proposition}\label{prop_seshomology}
    Let $\X$ be a mCRS. Then, there is a (not canonically) split short exact sequence
\begin{equation}\label{ses:homology_groups}\begin{tikzcd}
	0 & {\bigoplus_{v \in V} H_1(X_{v})} & {H_1(\mathfrak{X})} & {H_1(\Gamma)} & 0
	\arrow[from=1-1, to=1-2]
	\arrow["i_*", from=1-2, to=1-3]
	\arrow["j_*",from=1-3, to=1-4]
	\arrow[from=1-4, to=1-5]
\end{tikzcd},\end{equation}
where $i_*$ is the direct sum of the homology maps $H_1(X_v) \to H_1(\X)$ induced by inclusion, and $j_*$ is induced by the quotient map $\X \to \Gamma$.
\end{proposition}

\begin{remark}
    The following proof is due to an anonymous referee, who read an earlier version of this paper, and replaces our previous more cumbersome proof.
\end{remark}

\begin{proof}
    We construct a $CW$-decomposition of $\X$ in the following manner: Choose basepoints $p_v \in X_v$ for each $v \in V$. The $p_v$ together with the marked points will be our $0$-cells. Now, for each $X_v$, if $X_v$ has genus $g$, choose $2g$ closed $1$-cells on $X_v$ with basepoint $p_v$ as they appear in the standard $CW$-decomposition of a genus $g$ surface. Furthermore, choose $|\mathcal A_v|$ additional $1$-cells connecting each marked point to $p_v$ and also make the edges of $\Gamma$ $1$-cells in the obvious way. Finally, add one $2$-cell for each Riemann surface $X_v$. This gives the desired $CW$-decomposition. Note that after forgetting the $1$-cells that are needed to construct each $X_v$, the $1$-skeleton of this decomposition becomes homeomorphic to $\Gamma$. This gives an inclusion $ \iota : \Gamma \hookrightarrow |\X|$, and if $j : |\X| \to \Gamma$ is the quotient map given by collapsing each Riemann surface to a vertex of $\Gamma$, the composition of cellular maps
    $$
    \Gamma \overset{\iota}{\hookrightarrow} |\X| \overset{j}{\rightarrow} \Gamma 
    $$
    is homotopic to the identity. Taking homology, we see that $j_*$ is split surjective. Note however that $\iota_*$ is not canonical. 

    Since any $1$-cell of $X_v$ gets mapped to a point under the composition $j_* \circ i_*$, we see that \ref{ses:homology_groups} is a complex. Furthermore, suppose we have a cycle of $1$-cells in $\X$ that becomes exact after applying $j_*$. As $\Gamma$ has no $2$-cycle, the cycle is in fact trivial as a chain on $\Gamma$. Therefore, the original cycle cannot be supported on any $1$-cell corresponding to an edge of $\Gamma$ and thus has to come from the $X_v$.

    Finally, we show that \ref{ses:homology_groups} is also exact on the left. Consider the good pair $(\X, \bigcup_{v \in V} X_v)$ and apply the long exact homology sequence to obtain
    $$
    \hdots \to H_2(\X,\bigcup_{v \in V} X_v) \rightarrow  H_1(\bigcup_{v \in V} X_v) \overset{i_*}{\rightarrow} H_1(\X) \to \hdots.
    $$
    Noting that all $2$-cells in the $CW$-decomposition of $\X$ are contained in $\bigcup_{v \in V} X_v$, gives that $H_2(\X,\bigcup_{v \in V} X_v) = 0$ as desired.
\end{proof}

\begin{definition}
    Let $\X$ be a mCRS. We define the $\R$-vector space $\Omega(\X)$ of \emph{hybrid $1$-forms} to be
    $$
    \Omega(\X) := \left( \bigoplus_{v \in V} \Omega(X_v) \right) \oplus \Omega(\Gamma).
    $$
    Its dual $\Omega^*(\X)$ is defined to be
    $$
    \Omega^*(\X) := \left( \bigoplus_{v \in V} \Omega^*(X_v) \right) \oplus \Omega^*(\Gamma),
    $$
    where $\Omega^*(X_v) = \Hom\big(\Omega(X_v),\C\big)$ and $\Omega^*(\Gamma) = \Hom\big(\Omega(\Gamma),\R\big)$. 
\end{definition}

We obtain the canonically split sequence
\begin{center}
\begin{tikzcd}
	0 & {\bigoplus_{v \in V} \Omega^*(X_v)} & {\Omega^*(\mathfrak{X})} & {\Omega^*(\Gamma)} & 0
	\arrow[from=1-2, to=1-3]
	\arrow[from=1-3, to=1-4]
	\arrow[from=1-4, to=1-5]
	\arrow[from=1-1, to=1-2].
\end{tikzcd}
\end{center}

\begin{definition}
    A path/singular $1$-simplex $\gamma : [0,1] \to \X$ is called \textit{good} if $\gamma|_{\gamma^{-1}(X_v)}$ is piecewise smooth for all $v \in V$, and $j_*(\gamma)$ is piecewise smooth, too. We denote the subgroup of $C_1(\X)$ generated by all good paths by $GC_1(\X)$ and call $dC_2(\X) \cap GC_1(\X)$ the \textit{group of good boundaries}. Elements of either subgroup will also be called good.
\end{definition}

Using a standard argument (and, in particular, the existence of the Lebesgue number), one can see that every homology class $[\gamma] \in H_1(\X)$ has a good representative. 
Therefore, we obtain an isomorphism
\begin{equation}\label{eq_homologywithgoodpaths}
\frac{\ker\big(d : GC_1(\X) \rightarrow C_0(\X)\big)}{GC_1(\X) \cap dC_2(\X)} \cong H_1(\X).
\end{equation}
We adopt the notation $\int_\sigma$ for the linear functional 
$$
\omega \longmapsto \int_\sigma \omega
$$
associated to a good path/singular $1$-simplex $\sigma$ on $\X$.

\begin{proposition}\label{prop:map_of_ses_to_define_jacobian}
    There is a well-defined homomorphism
    $$
    \chi : H_1(\X) \longrightarrow \Omega^*(\X)
    $$
    such that the following diagram commutes
\begin{center}\begin{tikzcd}
	0 & {\bigoplus_{v \in V} H_1(X_v)} & {H_1(\mathfrak{X})} & {H_1(\Gamma)} & 0 \\
	0 & {\bigoplus_{v \in V} \Omega^*(X_v)} & {\Omega^*(\mathfrak{X})} & {\Omega^*(\Gamma)} & 0
	\arrow[from=2-2, to=2-3]
	\arrow[from=2-3, to=2-4]
	\arrow[from=2-4, to=2-5]
	\arrow[from=2-1, to=2-2]
	\arrow[from=1-4, to=2-4]
	\arrow["\chi",from=1-3, to=2-3]
	\arrow[from=1-2, to=2-2]
	\arrow["i_*",from=1-2, to=1-3]
	\arrow["j_*",from=1-3, to=1-4]
	\arrow[from=1-4, to=1-5]
	\arrow[from=1-1, to=1-2].
\end{tikzcd}\end{center}
\end{proposition}

\begin{proof}
    We start by constructing a homomorphism $\chi : GC_1(\X) \rightarrow \Omega^*(\X)$. Let $n := |V|$. For $\sum_i \sigma_i \in GC_1(\X)$, we define $\chi$ by
    $$
    \sum_i \sigma_i \mapsto \left(\underbrace{\left(\sum_i \int_{\sigma_i \cap X_{v_1}},\hdots,\sum_i \int_{\sigma_i \cap X_{v_n}}\right)}_{\chi_X :=},\underbrace{\left(\sum_i \int_{j_*(\sigma_i)}\right)}_{\chi_\Gamma :=}\right),
    $$
    where 
    $$
    \int_{\sigma_i \cap X_{v_i}}
    $$
    means that we only integrate along the parts of $\sigma_i$ that are in $X_{v_i}$. This definition shows the importance of introducing good paths for otherwise the various integrals would not be well-defined.

    We claim that, using the isomorphism \eqref{eq_homologywithgoodpaths}, the homomorphism $\chi$ descends to a map $H_1(\X) \to \Omega^*(\X)$. Using linearity, it will be enough to show $\chi(d\sigma) = 0$ for all good boundaries $d\sigma \in GC_1(\X) \cap dC_2(\X)$, and we can show this for the $\chi_\Gamma$-part and the $\chi_X$-part of $\chi$ separately.
    
    Note that for the $\chi_\Gamma$-part, this is an immediate consequence of the analogous statement for metric graphs and the commutativity of
    \begin{center}
    \begin{tikzcd}
    	{C_2(\mathfrak{X})} & {C_1(\mathfrak{X})} \\
    	{C_2(\Gamma)} & {C_1(\Gamma)}
    	\arrow["d", from=1-1, to=1-2]
    	\arrow["d",from=2-1, to=2-2]
    	\arrow["{j_*}"', from=1-1, to=2-1]
    	\arrow["{j_*}"', from=1-2, to=2-2]
    \end{tikzcd}
    \end{center} 
    as these statements together imply
    $$
    \chi_\Gamma(d\sigma) = \int_{j_*(d\sigma)} = \int_{d(j_*(\sigma))} \equiv 0.
    $$
    We now consider the $\chi_X$-part. Instead of working with arbitrary good boundaries, we want to reduce the situation to good boundaries already contained in smaller open subsets of $\X$ that are easier to study. To this end, choose an open covering $\mathcal{U}$ of $\X$ by simply connected open sets such that for every $x_v^e$ there is exactly one $U_{v,e} \in \mathcal{U}$ with $x_v^e \in U_{v,e}$. It is possible to choose $\mathcal{U}$ in such a way that the sets $U_{v,e} \cap X_v$ are pairwise disjoint and each of them is a connected domain of some chart of $X_v$. 
    The preimages of these sets under $\sigma$ give an open cover $\mathcal{U}'$ of the standard simplex $\Delta^2 = \mathrm{conv}\{e_0,e_1,e_2\}$. 

    We can find a barycentric subdivision of $\Delta^2$ such that the following two conditions hold for all simplices $\widetilde{\Delta}^2$ of the subdivision:
    \begin{itemize}
        \item[(1)] $\sigma(\widetilde{\Delta}^2) \subseteq U$ for some $U \in \mathcal{U}$.
        \item[(2)] If a point of $\widetilde{\Delta}^2$ is mapped to $x_v^e$ by $\sigma$, $\sigma(\widetilde{\Delta}^2)$ is contained in $U_{v,e}$.
    \end{itemize}
    This is an easy consequence of the existence of a Lebesgue number for the covering $\mathcal{U}'$ of the compact set $\Delta^2$.
    
    We now have to deal with the technical issue that the paths associated to the edges of the simplices of the subdivision are no longer necessarily good (unless the edge in question is already part of an edge of the big simplex $\Delta^2$). Therefore, these edges have to be modified slightly if we want to integrate along them. An inner edge $k$ with endpoints $P$ and $Q$ is part of two $2$-simplices $\widetilde{\Delta}^2_1$ and $\widetilde{\Delta}^2_2$ of the subdivision, and the orientations of $k$ induced by applying the boundary operator $d$ to $\widetilde{\Delta}_1^2$ and $\widetilde{\Delta}^2_2$ are opposite to each other. Let $U_1,U_2 \in \mathcal{U}$ be two open sets such that $\widetilde{\Delta}^2_i \subseteq U_i, \ i=1,2$. Their existence is guaranteed by (1). Note that the path $\sigma_k$ is located entirely in $U_1 \cap U_2$ and connects $\sigma(P)$ and $\sigma(Q)$. Choose one of the two possible orientations of $k$ to obtain an oriented edge from $P$ to $Q$ (or, alternatively, from $Q$ to $P$). We again denote this oriented edge by $k$ and write $\overline{k}$ for the same edge with its orientation reversed. It is now clearly possible to choose a good path $\gamma_k$ which connects $\sigma(P)$ to $\sigma(Q)$ and is still located entirely in $U_1 \cap U_2$. We set $\gamma_{\overline{k}} := \overline{\gamma}_k$ where $\overline{\gamma}_k$ is the inverse path to $\gamma_k$. For an outer edge of the subdivion, whose associated path in $\X$ is already good, we simply set $\gamma_k := \sigma|_k$, where the orientation of $\gamma_k$ is induced by the orientation of the big outer edge of which $k$ is a part. We now have
    \begin{align*}
        \chi(d \sigma) &= \sum_{\substack{k \\ \textnormal{outer edge}}} \chi(\gamma_k) + \sum_{\substack{k \\ \textnormal{inner edge}}} \left( \chi(\gamma_k) + \chi(\gamma_{\overline{k}})\right) \\
        &= \sum \chi(\gamma_{k_1} \cdot \gamma_{k_2} \cdot \gamma_{k_3}),
    \end{align*}
    where $\cdot$ denotes the composition of paths and the last sum is taken over all $2$-simplices $\widetilde{\Delta}^2$ of the subdivision with oriented edges $k_1,k_2$ and $k_3$.

    \begin{figure}[ht]
        \centering
        \begin{tikzpicture}[x=0.9cm,y=0.9cm]
    
            \draw[rounded corners,thick,->,postaction={
                decorate,
                decoration={
                    markings,
                    mark=at position 0.5 with \coordinate (z);
                }
            }] (-1,7) .. controls (0.66,7.5) and (2.33,7.5) .. (4,7);
            \draw (z) node[above] {$\sigma$};
              
            \begin{scope}[xshift=5cm,yshift=0cm]
            \draw[very thick,rounded corners=3mm] (0.5,0.5)--(1.5,1.2)--(3,1)--(3.1,3)--(3.5,4)--(3.2,5)--(0,6.5)--(-2,6.3)--(-3.8,4)--(-4,2)--cycle;
            \draw (2.5,6) node[right] {$\X$};
            \draw (2,4.2) node[right] {$U_1$};
            \draw (2,2.5) node[right] {$U_2$};
            \draw[dotted,thick] (0,3) circle [radius=2];
            \draw[dotted,thick] (0,4) circle [radius=2];
            \fill[font=\tiny] (-1,3.5) circle (1.5pt) node[above] {$\sigma(P)$} (1,3.5) circle (1.5pt) node[above] {$\sigma(Q)$};
            \draw (-1,3.5) -- (-0.1,3) -- (0,3.9) -- (0.5,3.2) -- (1,3.5);
            
            \path [rounded corners,thick,draw=black,postaction={on each segment={mid arrow=black}}] (-1,3.5) .. controls (-0.5,3.2) and (0.5,3.9) .. (1,3.5);
            
            \end{scope}
      
          \coordinate (first) at (-3,6.9282);
          \coordinate (second) at (-7,0);
          \coordinate (third) at (1,0);
        
          \node [node font=\tiny, above]       at (first) {$e_2$};
          \node [node font=\tiny, below left]  at (second) {$P = e_0$};
          \node [node font=\tiny, below right] at (third) {$e_1$};
          
          \draw [line width=1pt,black] (first.south) -- (second.north east) -- (third.north west) -- cycle;
          \draw [line width=1pt,black] (first) -- (barycentric cs:first=0,second=1,third=1);
          \draw [line width=1pt,black] (second) -- (barycentric cs:first=1,second=0,third=1);
          \draw [line width=1pt,black] (third) -- (barycentric cs:first=1,second=1,third=0);
        
          \node [node font=\tiny, above left] at (barycentric cs:first=1.2,second=1,third=1) {Q};
          \coordinate (middle) at (barycentric cs:first=1,second=1,third=1);
          \coordinate (first1) at (barycentric cs:first=0,second=1,third=1);
          \coordinate (second1) at (barycentric cs:first=1,second=0,third=1);
          \coordinate (third1) at (barycentric cs:first=1,second=1,third=0);
    
            %Drawing k
            \draw[line width=1pt,-{Stealth}] (barycentric cs:first=0,second=1,third=0) -- (barycentric cs:second=1,middle=3.2,third1=0);
            \node [node font=\tiny, above left] at (barycentric cs:second=1,middle=3.3,third1=0.1) {$k$};
            
            %Drawing \overline{k}
            \draw[line width=1pt,-{Stealth}] (middle) -- (barycentric cs:second=2.9,middle=1,third1=0);
            \node [node font=\tiny,below right] at (barycentric cs:second=2.8,middle=1.2,first1=0.1) {$\overline{k}$};
            
            %Drawing \widetilde{\Delta}^2_1
            \node [node font=\tiny] at (barycentric cs:first=0.2,second=1,third=0.5) {$\widetilde{\Delta}^2_1$};
            
            %Drawing \widetilde{\Delta}^2_2
            \node [node font=\tiny] at (barycentric cs:first=0.5,second=1,third=0.2) {$\widetilde{\Delta}^2_1$};
            
            %Drawing arrows on other sides
            \draw[line width=1pt,-{Stealth}] (middle) -- (barycentric cs:second=0,middle=1,third1=2);
            
            \draw[line width=1pt,-{Stealth}] (third1) -- (barycentric cs:second=0,middle=2,third1=1);
            
            \path [draw=black,postaction={on each segment={mid arrow=black}}]
            (first) -- (third1);
            
            \path [draw=black,postaction={on each segment={mid arrow=black}}]
            (third1) -- (second);
            
            \path [draw=black,postaction={on each segment={mid arrow=black}}]
            (second) -- (first1);
            
            \path [draw=black,postaction={on each segment={mid arrow=black}}]
            (first1) -- (third);
            
            \path [draw=black,postaction={on each segment={mid arrow=black}}]
            (third) -- (second1);
            
            \path [draw=black,postaction={on each segment={mid arrow=black}}]
            (third) -- (second1);
            
            \path [draw=black,postaction={on each segment={mid arrow=black}}]
            (second1) -- (first);
            
            %Drawing arrows on remaining inner lines
            \draw[line width=1pt,-{Stealth}] (barycentric cs:first=1,second=0,third=0) -- (barycentric cs:first=1,middle=3.2,second1=0);
            
            \draw[line width=1pt,-{Stealth}] (middle) -- (barycentric cs:first=2.9,middle=1,third1=0);
            
            \draw[line width=1pt,-{Stealth}] (barycentric cs:first=0,second=0,third=1) -- (barycentric cs:third=1,middle=3.2,second1=0);
            
            \draw[line width=1pt,-{Stealth}] (middle) -- (barycentric cs:third=2.9,middle=1,third1=0);
            
            \draw[line width=1pt,-{Stealth}] (middle) -- (barycentric cs:first=0,middle=1,second1=2);
            
            \draw[line width=1pt,-{Stealth}] (second1) -- (barycentric cs:first=0,middle=2,second1=1);
            
            \draw[line width=1pt,-{Stealth}] (middle) -- (barycentric cs:second=0,middle=1,first1=2);
            
            \draw[line width=1pt,-{Stealth}] (first1) -- (barycentric cs:second=0,middle=2,first1=1);
            
            \end{tikzpicture}
        \caption{Barycentirc subdivision of $\Delta^2$ under the assumption that one iteration is already sufficient. The zig-zag line from $\sigma(P)$ to $\sigma(Q)$ depicts $\sigma|_k$ whereas the smooth path depicts a possible choice of $\gamma_k$.}
    \end{figure}
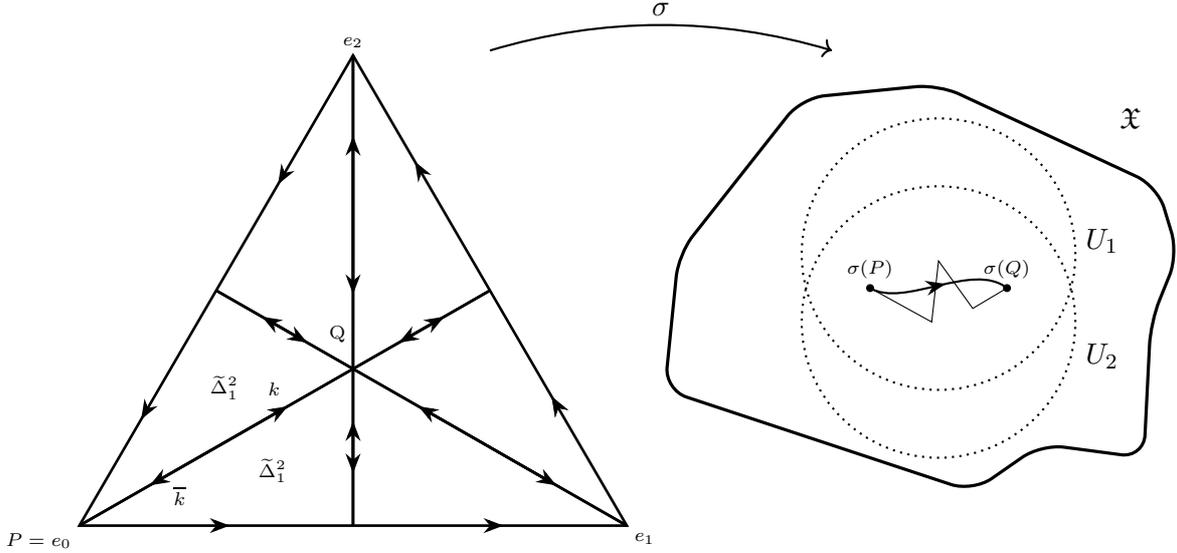

    We need to show that each summand in the last sum is already trivial. By construction, the path $\gamma := \gamma_{k_1} \cdot \gamma_{k_2} \cdot \gamma_{k_3}$ is a good closed path in one of the $ U \in \mathcal{U}$. Since $U$ is simply connected, $\gamma$ is null-homotopic. There are three possible cases:
    \begin{itemize}
        \item $U \subseteq X_v$ for some $v \in V$. Then, $\chi_X(\gamma) = 0$ as path integrals on Riemann surfaces are homotopy invariant and $\gamma$ is null-homotopic.
        \item $U \subseteq \X \setminus \bigcup_{v \in V} X_v$. In this case, $\chi_X(\gamma) = 0$ is trivially true since $\gamma \cap X_v = \emptyset$ for all $v \in V$.
        \item $U = U_{v,e}$ for some $v \in V$ and $e \in E$. Let $Y$ be the topological space obtained by taking an open simply connected neighborhood $V \subseteq \C$ of $0$ and gluing the interval $[0,\epsilon), \ \epsilon > 0$ to it by identifying the $0$ points with each other. As the $U_{v,e}$ can be chosen to be of this exact form, we have now reduced the problem to considering null-homotopic good paths $\gamma : [0,1] \to Y$. We have to show that $\gamma \cap V$ can also be parametrized by a good null-homotopic path. As in the first case, this will imply
        $$
        \int_{\gamma \cap V} \equiv 0.
        $$
        Without loss of generality, we may assume that $\gamma$ is not entirely contained in $(0,\epsilon)$. Otherwise, $\gamma \cap V = \emptyset$ and the above integral is identically $0$ for trivial reasons. Let $\pi : Y \to Y / [0,\epsilon) \cong V$ be the quotient map. Then, $\gamma \cap V$ is parametrized by $\pi_*(\gamma)$ and this path is null-homotopic since $V$ is simply connected.
    \end{itemize}

    The commutativity of the diagram is obvious from the definition of $\chi$.
\end{proof}

Applying the snake lemma, we obtain the following short exact sequence
\begin{equation}\label{eq:ses_of_jacobians}
\begin{tikzcd}
	0 & {\bigoplus_{v \in V}\Jac(X_v)} & {\Omega^*(\mathfrak{X})/H_1(\mathfrak{X})} & {\Jac(\Gamma)} & 0
	\arrow[from=1-2, to=1-3]
	\arrow[from=1-3, to=1-4]
	\arrow[from=1-4, to=1-5]
	\arrow[from=1-1, to=1-2].
\end{tikzcd}\end{equation}
This motivates our definition of the Jacobian of $\X$.

\begin{definition}[Jacobian of a mCRS]
    We call
    $$
    \Jac(\X) := \Omega^*(\X)/H_1(\X)
    $$
    the Jacobian of the mCRS $\X$.
\end{definition}

\begin{remark}
    Our definition is compatible with the definition of the Jacobian of Riemann surfaces and metric graphs: If $\X$ is essentially a Riemann surface $X$ in the sense of Example \ref{ex:Riemann_surfaces_as_mCRS}, the Jacobian of $\X$ and the Jacobian of $X$ are canonically isomorphic. If $\X$ is essentially a metric graph in the sense of Example \ref{ex:metric_graphs_as_mCRS}, we likewise have $\Jac(\X) \cong \Jac(\Gamma)$.
\end{remark}

\begin{remark}
    Since the group on the left of \eqref{eq:ses_of_jacobians} is divisible, the short exact sequence splits and we conclude that abstractly
    $$
    \Jac(\X) \cong \bigoplus_{v \in V} \Jac(X_v) \oplus \Jac(\Gamma).
    $$
    Given that we also have 
    $$
    H_1(\X) \cong \bigoplus_{v \in V} H_1(X_v) \oplus H_1(\Gamma),
    $$
    one might be tempted to define $\chi$ in Proposition \ref{prop:map_of_ses_to_define_jacobian} as the direct sum of
    $$
    \bigoplus_{v \in V} H_1(X_v) \hookrightarrow \bigoplus_{v \in V} \Omega^*(X_v)
    $$
    and
    $$
    H_1(\Gamma) \hookrightarrow \Omega^*(\Gamma)
    $$
    so that one can take
    $$
    \left(\bigoplus_{v \in V} \Omega^*(X_v) \oplus \Omega^*(\Gamma) \right) \Big/ \left( \bigoplus_{v \in V} H_1(X_v,\Z) \oplus H_1(\Gamma,\Z) \right) = \bigoplus_{v \in V} \Jac(X_v) \oplus \Jac(\Gamma)
    $$
    directly as a definition for the Jacobian of a mCRS. However, since the splitting of the homology groups depends on the choice of an embedding $\Gamma \hookrightarrow |\X|$, our Jacobian is in fact not canonically split. The advantage of our definition is that $\chi$ is defined in a canonical way that allows for the direct construction of an Abel-Jacobi map, as carried out in the next section. 
\end{remark}

\section{Abel--Jacobi theory for metrized complexes of Riemann surfaces}

Fix $\X$ to be a mCRS. From now on, all paths in $\X$ are assumed to be good. We therefore drop the word \emph{good} and simply talk about paths.

\begin{definition}[Abel-Jacobi map]
    Choose a basepoint $p_0 \in \X$. For a point $p \in \X$, choose a path $\gamma_p$ from $p_0$ to $p$. We define the \emph{Abel-Jacobi map}
    \begin{equation*}\begin{split}
    A_\X : \X &\longrightarrow \Jac(\X)\\
    p &\longmapsto \left[\left(\int_{\gamma_p \cap X_{v_1}},\hdots, \int_{\gamma_p \cap X_{v_n}}\right),\left( \int_{j_*(\gamma_p)}\right)\right],
    \end{split}\end{equation*}
    where $j_*$ again denotes the pushforward along the quotient map $\X \to \Gamma$.
\end{definition}

The same argument as for the classical Abel-Jacobi map ensures that this map is well-defined, i.e. independent of the choice for $\gamma_p$. We call $A_\X$ the Abel-Jacobi map of $\X$.

As in the theory of Riemann surfaces, we can extend $A_\X$ linearly to a map $\Div(\X) \to \Jac(\X)$ by setting
$$
A_\X\Big(\sum n_p p\Big) := \sum n_p A_\X(p) \ .
$$
Its restriction to $\Div_0(\X)$, which we denote by $A_{\X,0}$, does no longer depend on the choice of $p_0$ (see \cite[Lemma VIII.2.1]{Miranda_RS} for an analogous argument in the case of Riemann surfaces).

We now start with the proof of Theorem \ref{mainthm_AbelJacobimCRS}, an Abel-Jacobi theorem for mCRS, which we restate here.

\begin{theorem}\label{thm:abel_jacobi_for_mCRS}
    There is an isomorphism of short exact sequences given by the following commutative diagram

\[\begin{tikzcd}
	0 & {\bigoplus_{v \in V}\mathrm{Pic}^0(X_v)} & {\mathrm{Pic}^0(\X)} & {\mathrm{Pic}^0(\Gamma)} & 0 \\
	0 & {\bigoplus_{v \in V} \Jac(X_v)} & {\Jac(\X)} & {\Jac(\Gamma)} & 0
	\arrow[from=1-1, to=1-2]
	\arrow[from=1-2, to=1-3]
	\arrow["A_{X,0} = \bigoplus_{v \in V} A_{v,0}",from=1-2, to=2-2]
	\arrow["\small{{[D] \mapsto [D_\Gamma]}}"',from=1-3, to=1-4]
	\arrow["{A_{\X,0}}", from=1-3, to=2-3]
	%\arrow["u"', curve={height=12pt}, dashed, from=1-4, to=1-3]
	\arrow[from=1-4, to=1-5]
	\arrow["{A_{\Gamma,0}}", from=1-4, to=2-4]
	\arrow[from=2-1, to=2-2]
	\arrow[from=2-2, to=2-3]
	\arrow[from=2-3, to=2-4]
	\arrow[from=2-4, to=2-5]
\end{tikzcd}.\]
Furthermore, both rows are right-split, and the splitting only depends on a set of basepoints $\mathcal{B}$, where we choose one basepoint $p_v \in X_v$ for each Riemann surface $X_v$. The splitting map in the top row is then given by
$$
u : \Pic^0(\Gamma) \to \Pic^0(\X), \ \left[\sum_j a^j\right] \mapsto \left[\sum_j b^j\right],
$$
where
$$
b^j = \begin{cases}
    a^j & \textnormal{if } a^j \in \Gamma \setminus V \\
    p_v & \textnormal{if } a^j = v \in V.
\end{cases}
$$
\end{theorem}

\begin{remark}
    It is not immediately obvious that $u$ is well-defined. Our strategy consists in first establishing the existence of the diagram and then using the bottom row to show that $u$ is indeed well-defined.
\end{remark}

The following proposition is one of three key statements needed for the proof of Theorem \ref{thm:abel_jacobi_for_mCRS}.

\begin{proposition}\label{prop:abel_jacobi_maps_commute}
    The Abel-Jacobi maps induce a commutative diagram:
    \begin{equation}\label{eq_twoses}\begin{tikzcd}
    	0 & {\bigoplus_{v \in V}\mathrm{Div_0}(X_v)} & {\mathrm{Div_0(\mathfrak{X})}} & {\mathrm{Div_0(\Gamma)}} & 0 \\
    	0 & {\bigoplus_{v \in V}\Jac(X_v)} & {\Jac(\mathfrak{X})} & {\Jac(\Gamma)} & 0
    	\arrow[from=1-1, to=1-2]
    	\arrow[from=1-2, to=1-3]
    	\arrow["{D \mapsto D_\Gamma}", from=1-3, to=1-4]
    	\arrow[from=1-4, to=1-5]
    	\arrow[from=2-1, to=2-2]
    	\arrow[from=2-2, to=2-3]
    	\arrow[from=2-3, to=2-4]
    	\arrow[from=2-4, to=2-5]
    	\arrow["A_{X,0} = \bigoplus_{v \in V} A_{v,0}",from=1-2, to=2-2]
    	\arrow["A_{\mathfrak{X},0}", from=1-3, to=2-3]
    	\arrow["A_{\Gamma,0}",from=1-4, to=2-4].
    \end{tikzcd}\end{equation}
\end{proposition}

\begin{proof}
    The exactness of the first line is immediate. 
    To show commutativity, we choose a base point $p_0 \in X_{v_0}$ for some $v_0 \in V$ and a spanning tree $T$ of $\Gamma$. 
    By linearity, it suffices to consider an element $D_{v_1}$ of $\bigoplus_{v \in V} \Div_0(X_v)$ which is only nonzero in the component of $X_{v_1}$. There is exactly one sequence of edges $(e_1,\hdots,e_m)$ in $T$ connecting $v_0$ to $v_1$. We choose $x_{v_1}^{e_m}$ as the basepoint for the Abel-Jacobi map $A_{v_1,0} : \Div_0(X_{v_1}) \to \Jac(X_{v_1})$. Since $\deg(D_{v_1}) = 0$, we can write $D_{v_1}$ as
    $$
    D_{v_1} = \sum_j \big(a_j^1 - a_j^2\big)
    $$
    for suitably chosen, not necessarily distinct points $a_j^i \in X_{v_1}$. Again by linearity, it suffices to show
    \begin{equation}\label{eq:commutativity_riemmann_surfaces_mCRS}
        A_{\X,0}(a_j^1-a_j^2) = \big[(0,\hdots,0,A_{v_1,0}(a_j^1-a_j^2),0,\hdots,0),0\big].
    \end{equation}
    As $A_{\X,0}(p)$ is independent of the chosen path from $p_0$ to $p$ for all $p \in \X$, we can assume without loss of generality that the chosen paths from $p_0$ to $a_j^1$ and $a_j^2$ follow the edge sequence $(e_1,\hdots,e_m)$ and only begin to differ from each other at $x_{v_1}^{e_m}$. Furthermore, we can assume that the restriction of these paths to a path from $x_{v_1}^{e_m}$ to $a_j^1$ and $a_j^2$, respectively, is contained in $X_{v_1}$ and identical to the paths chosen for the maps $a_j^i \mapsto A_{v_1,0}(a_j^i)$ for both $i=1,2$. As the images of $a_j^1$ and $a_j^2$ need to be considered with opposite signs, all contributions to components of $A_{\X,0}$ other than the $X_{v_1}$-component cancel each other out. This shows \eqref{eq:commutativity_riemmann_surfaces_mCRS}.

    Now, let $D \in \Div_0(\X)$. Linearity allows us to assume $D = a^1-a^2$. There are two possible cases to consider:
    \begin{itemize}
        \item[(a)] Suppose $a^1,a^2 \in X_v$ for some $v \in V$. Then, $D_\Gamma = 0$ and (with the necessary identification) $D \in \Div_0(X_v)$. By what we already proved, it follows that
        $$
        A_{\X,0}(D) = \big[(0,\hdots,0,A_{v,0}(D),0,\hdots,0),0\big].
        $$
        Therefore, the projection to the $\Gamma$-part of $A_{\X,0}(D)$ is equal to $0$, and similarly we have
        $$
        A_{\Gamma,0}(D_\Gamma) = A_{\Gamma,0}(0) = 0 \in \Jac(\Gamma).
        $$
        \item[(b)] If we are not in case (a), we have $D_\Gamma = j(a^1)-j(a^2)$, where $j : \X \to \Gamma$ is the quotient map. Choose a path $\gamma \subseteq \X$ form $a^2$ to $a^1$ and paths $\alpha_i \subseteq \X$ from $p_0$ to $a^i$. Then, $\alpha_1 - \gamma - \alpha_2$ is a closed cycle in $\X$ and $j_*(\alpha_1 - \gamma - \alpha_2) = j_*(\alpha_1)-j_*(\gamma)-j_*(\alpha_2)$ is a closed cycle in $\Gamma$. Therefore, there is a representative of $A_{\X,0}(D)$ with $\Gamma$-part
        $$
        \int_{j_*(\alpha_1)} - \int_{j_*(\alpha_2)} = \int_{j_*(\gamma)}.
        $$
        As $j_*(\alpha_i)$ is a path from $j(p_0)$ to $j(a^i)$, we also have
        $$
        A_{\Gamma,0}(D_\Gamma) =  \left[\int_{j_*(\alpha_1)} \right] - \left [\int_{j_*(\alpha_2)} \right] = \left[\int_{j_*(\gamma)} \right],
        $$
        as desired.
    \end{itemize}
\end{proof}

Two prove our next proposition, we make use of \cite[Lemma 2]{HaaseMusikerYu} which shows that rational functions on metric graphs can be written as finite sums of simple functions, so-called \textit{weighted chip-firing moves}. Here we recall the necessary definitions:

\begin{definition}[Subgraph]
    Let $\Gamma$ be a metric graph. A \textit{subgraph} $\Gamma'$ is a compact subset with a finite number of connected components.
\end{definition}

\begin{definition}[Weighted chip-firing move]\label{def:weighted_chip_firing_move}
    Let $\Gamma$ be a compact metric graph. Let $\Gamma_1$ and $\Gamma_2$ be two subgraphs such that $\Gamma_1 \cap \Gamma_2 = \emptyset$ and $\Gamma \setminus (\Gamma_1 \cup \Gamma_2)$ only contains points of valency $2$. Then, $f \in \Rat(\Gamma)$ with
    \begin{itemize}
        \item $f|_{\Gamma_i}$ is constant
        \item $f|_{\Gamma \setminus (\Gamma_1 \cup \Gamma_2)}$ is linear with integer slopes on each connected component
    \end{itemize}
    is called \textit{weighted chip-firing move}.
\end{definition}

\begin{lemma}[\cite{HaaseMusikerYu} Lemma 2] \label{lemma:weighted_chip_firing_moves}
    Every rational function on $\Gamma$ can be written as a finite sum of weighted chip-firing moves.
\end{lemma}

\begin{proposition}\label{prop:extended_rational_function_maps_to_zero}
    Let $f_\Gamma \in \Rat(\Gamma)$ be a rational function. By choosing arbitrary nonzero constants for each Riemann surface $X_v$, we can extend $f_\Gamma$ to a rational function $F_\Gamma \in \Rat(\X)$ on $\X$ with $\div(F_\Gamma) \in \Div_0(\X)$ and $\div(F_\Gamma)_\Gamma = \div(f_\Gamma)$. We then have
    $$
    A_{\X,0}\big(\div(F_\Gamma)\big) = 0.
    $$
\end{proposition}

\begin{proof}
    By the linearity of $A_{\X,0}$ and Lemma \ref{lemma:weighted_chip_firing_moves}, we can assume that $f_\Gamma$ is a weighted chip-firing move. Therfore, there exist subgraphs $\Gamma_1$ and $\Gamma_2$ as in Definition \ref{def:weighted_chip_firing_move} such that
    $$
    \overline{\Gamma \setminus (\Gamma_1 \cup \Gamma_2)} = L_1 \sqcup \hdots \sqcup L_r.
    $$
    Here, the $L_j$ are closed line segments with endpoints $a_j^1 \in \Gamma_1$ and $a^2_j \in \Gamma_2$ such that 
    $f_\Gamma|_{L_j}$
    is a linear function with integer slope $m_j$, and $f|_{\Gamma_1}$ and $f|_{\Gamma_2}$ each are constant. We may assume $f_\Gamma(\Gamma_1) < f_\Gamma(\Gamma_2)$.

    \begin{figure}[ht]
\centering
 \begin{tikzpicture}[x=1.5cm,y=1.5cm]

    \draw (0,0) circle (1);
        \draw (-1,0) .. controls (-1,-0.3) and (1,-0.3) .. (1,0);
        \draw [dotted] (-1,0) .. controls (-1,0.3) and (1,0.3) .. (1,0);

    \begin{scope}[xshift=3.75cm, yshift=-2.25cm]
    \draw (0,-0.1) ellipse (1.5 and 0.7);
    \draw [name path=test1] (-0.5,0) .. controls (-0.5,-0.25) and (0.5,-0.25) .. (0.5,0);
    \draw (-0.5,0) .. controls (-0.5,-0.25) and (0.5,-0.25) .. (0.5,0);
    \path [name path=test2] (-0.5,-0.1) -- (0.5,-0.1);
    \draw [name intersections={of=test1 and test2}]
    (intersection-1) .. controls (intersection-1 |- 0,0.1) and (intersection-2 |- 0,0.1) .. (intersection-2); 
    \end{scope}

    \path % let's define some points:
      coordinate (p1) at (-2,1.4)
      coordinate (p2) at (-0.5,0.5)
      coordinate (p3) at (0.4,-0.04)
      coordinate (p4) at (2.2,-1.06)
      coordinate (p5) at (3.3,-1.72)
      coordinate (p6) at (4.5,-2.44)
      coordinate (p7) at (5.5,-3.04)
      coordinate (p8) at (6.5,-3.64)
      coordinate (p9) at (-2,3.4)
      coordinate (p10) at (6.5,-1.64)
      coordinate (p11) at (4.5,-0.44)
      coordinate (p12) at (5.5,-1.04);
      
  \node[node font=\scriptsize, left] at (p1) {$a_j^2$};
  \node[node font=\tiny, right] at (p2) {$x_{v_j^{l_j}}^{e_j^{l_j+1}}$};
  \node[node font=\tiny, above] at (p3) {$x_{v_j^{l_j}}^{e_j^{l_j}}$};
  \node[node font=\tiny, above] at (2.2,-0.95) {$x_{v_j^{l_j-1}}^{e_j^{l_j}}$};
  \node[node font=\tiny, above] at (p5) {$x_{v_j^{l_j-1}}^{e_j^{l_j-1}}$};
  \node[node font=\tiny, above right] at (p6) {$x_{v_j^{l_j-2}}^{e_j^{l_j-1}}$};
  \node[node font=\tiny, above right] at (p7) {$x_{v_j^{1}}^{e_j^{1}}$};
  \node[node font=\scriptsize, right] at (p8) {$a_j^1$};

  \draw (p1) to node[node font=\tiny] [sloped,below] {$e_j^{l_j+1}$} (p2);
  \draw (p3) to node[node font=\tiny] [sloped,below] {$e_j^{l_j}$} (p4);
  \draw (p5) to node[node font=\tiny] [sloped,below] {$e_j^{l_j-1}$} (p6);
  \draw (p7) to node[node font=\tiny] [sloped,below] {$e_j^1$} (p8);
  \fill[black] (p1) circle [radius=2pt];
  \fill[black] (p2) circle [radius=2pt];
  \fill[black] (p3) circle [radius=2pt];
  \fill[black] (p4) circle [radius=2pt];
  \fill[black] (p5) circle [radius=2pt];
  \fill[black] (p6) circle [radius=2pt];
  \fill[black] (p7) circle [radius=2pt];
  \fill[black] (p8) circle [radius=2pt];
  \fill[black] (p9) circle [radius=2pt];
  \fill[black] (p10) circle [radius=2pt];
  \node at (5,-2.7) {$\dots$};

 %Gamma part
  \node[node font=\scriptsize, left] at (p9) {$a_j^2$};
  \node[node font=\scriptsize, right] at (p10) {$a_j^1$};
  \draw (p9) to node[node font=\small] [sloped,above] {$L_j \subseteq \Gamma$} (p11);
  \draw (p12) -- (p10);
  \node at (5,-0.7) {$\dots$};

  %Riemann Surfaces names
  \node[node font=\scriptsize] at (-1.4,-1) {$X_{v_j^{l_j}}$};
  \node[node font=\scriptsize] at (0.4,-2.2) {$X_{v_j^{l_j-1}}$};
    
\end{tikzpicture}
\caption{This picture shows the second case below. In the geometric realization $|\X|$ of $\X$, the line segment $L_j$ has Riemann surfaces $X_{v_j^i}$ with two marked points each in its interior.}
\end{figure}
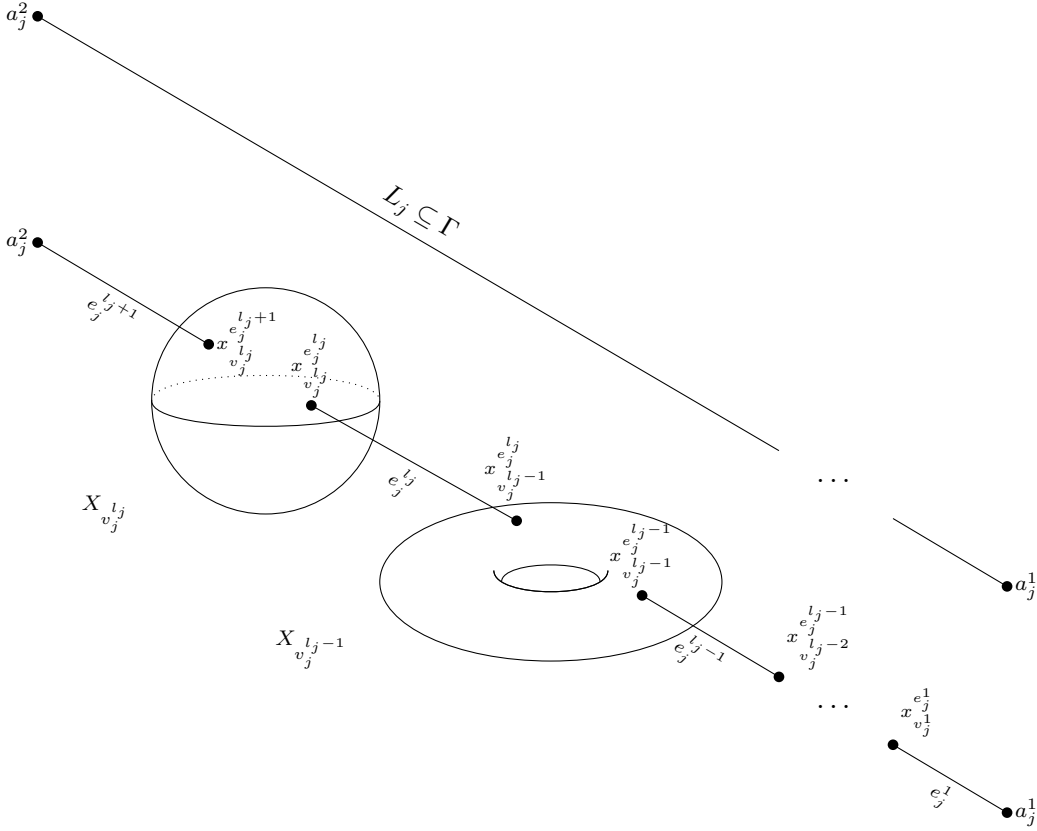

    The divisor of $f_\Gamma$ can be written as
    $$
    \div(f_\Gamma) = \sum_{j = 1}^r m_j a_j^1 - m_j a_j^2.
    $$
    We extend $f_\Gamma$ to a function $F_\Gamma \in \Rat(\X)$ as described in the statement of the proposition. Note that the interior of each line segment $L_j$ only contains points $p$ with $\val(p) = 2$. There are two possible cases:
    \begin{itemize}
        \item[(1)] The interior of $L_j$ does not contain any points which are vertices in the chosen model $G$ of $\Gamma$ in the definition of $\X$. Therefore, $L_j$, considered as a subset of $\X$, is entirely contained in an edge $e \subseteq \X$.
        \item[(2)] There are vertices $v_j^1,\hdots,v_j^{l_j}$ in the chosen model $G$ of $\Gamma$ with associated Riemann surfaces $X_{v_j^i}$ lying in the interior of $L_j$. The sets of marked points are then of the form
        $$
        \mathcal{A}_{v_j^i} = \Big\{x_{v_j^i}^{e_j^i}, x_{v_j^i}^{e_j^{i+1}}\Big\},
        $$
        where $e_j^1$ is the edge on which $a_j^1$ is located, and $e_j^{l_j+1}$ is the edge on which $a_j^2$ is located. All other edges $e_j^i$ connect the Riemann surface $X_{v_j^{i-1}}$ to $X_{v_j^i}$.
    \end{itemize}

    We divide the indices according to these two case into the sets $J_1$ and $J_2$. We can then write
    \begin{equation*}
        \mathrm{div}(F_\Gamma) = \sum_{j \in J_1} m_j a_j^1 - m_j a_j^2 + \sum_{j \in J_2} m_j \left(a_j^1 + \left(\sum_{i = 1}^{l_j} - x_{v_j^i}^{e_j^{i}} + x_{v_j^i}^{e_j^{i+1}}\right) - a_j^2\right),
    \end{equation*}
    where we keep writing $a_j^1$ and $a_j^2$ even if one of these points is a vertex of the model $G$ and could therefore also be denoted by $x_v^e$.
    
     Suppose first that $j \in J_1$. Choose a path $\gamma$ parameterizing $L_j$ with start point $a_j^2$ and endpoint $a_j^1$. We claim that
        $$
        A_{\X,0}(a_j^1-a_j^2) = \left[(0,\hdots,0),\int_\gamma\right].
        $$
        To prove this claim, choose a path $\alpha_2$ from a basepoint $p_0$ to $a_j^2$. The path $\alpha_1 := \alpha_2 \cdot \gamma$ given by concatenation of $\alpha_2$ and $\gamma$ is a path from $p_0$ to $a_j^1$. It follows that
        \begin{align*}
            A_{\X,0}(a_j^1-a_j^2) &= \left[\left(
            \int_{X_{v_1} \cap \alpha_1}, \hdots, \int_{X_{v_n} \cap \alpha_1}
            \right),\int_{j_*(\alpha_1)}\right] - 
            \left[\left(
            \int_{X_{v_1} \cap \alpha_2}, \hdots, \int_{X_{v_n} \cap \alpha_2}
            \right), \int_{j_*(\alpha_2)}
            \right] \\
            &= \left[\left( \int_{X_{v_1} \cap \gamma}, \hdots, \int_{X_{v_n} \cap \gamma} \right), \int_\gamma \right].
        \end{align*}
        Since $L_i$ is contained in an edge, every intersection $X_{v} \cap \gamma$ contains at most two points. Therefore, $A_{\X,0}(a_j^1-a_j^2)$ has trivial $X$-part.
        
         Suppose now that $j \in J_2$. Then, $A_{\X,0}$ equals the sum of the integrals along the line segments
        $$
        \Big(a_j^2, x_{v_j^{l_j}}^{e_j^{l_j+1}}\Big),\hdots,\Big(x_{v_j^1}^{e_j^{1}},a_j^1\Big).
        $$
        One can now apply the calculation from the first case to each of these line segments separately to see that we still have
        $$
        A_{\X,0}\left(a_j^1 + \left(\sum_{i = 1}^{l_j} - x_{v_j^i}^{e_j^{i}} + x_{v_j^i}^{e_j^{i+1}}\right) - a_j^2\right) = \left[(0,\hdots,0),\int_{\gamma_j}\right]
        $$
        where $\gamma_j$ parametrizes $L_j$ from $a_j^2$ to $a_j^1$.
    
    So far, we have shown
    $$
    A_{\X,0}(\div(F_\Gamma)) = \left[(0,\hdots,0),\sum_{j = 1}^r m_j \int_{\gamma_j}
    \right]
    $$
    where $\gamma_j$ parameterizes $L_j$ going from $a_j^2$ to $a_j^1$ as above. We conclude by proving that $\sum_{j = 1}^r m_j \int_{\gamma_j}$ is trivial as an elements of $\Omega^*(\Gamma)$. This shows $A_{\X,0}\big(\div(F_\Gamma)\big) = 0 \in \Jac(\X)$ as claimed.

    Choose a model $\widetilde{G} = (\widetilde{V},\widetilde{E})$ of $\Gamma$ such that $\widetilde{G}$ has a vertex at each $a_j^i$ and no vertices in the interior of any $L_j$. When choosing an orientation for $\widetilde{G}$, the edges $e_j \in \widetilde{E}$ corresponding to a line segment $L_j$ shall be oriented such that $e_j^- = a_j^1$ and $e_j^+ = a_j^2$. Let
    $$
        \omega = \sum_{e \in \widetilde{E}} \omega_e de \in \Omega(\widetilde{G}) \cong \Omega(\Gamma).
    $$
    It follows that
    \begin{align*}
        \sum_{j = 1}^r m_j \int_{\gamma_j} \omega &= \sum_{j = 1}^r \sum_{e' \in \widetilde{E}} \frac{f_\Gamma(a^2_j)-f_\Gamma(a^1_j)}{\mathrm{length}(L_j)} \omega_{e'} \int_{L_j} de' \\
        &= \sum_{e \in \widetilde{E}} \sum_{e' \in \widetilde{E}} \frac{f_\Gamma(e^+) - f_\Gamma(e^-)}{\mathrm{length}(e)} \omega_{e'} \int_e de' \\
        &= \sum_{e \in \widetilde{E}} \big[f_\Gamma(e^+)-f_\Gamma(e^-)\big] w_e 
        = \sum_{x \in \widetilde{V}} f_\Gamma(x) \left(\sum_{\substack{e \in \widetilde{E} \\ x = e^+}} w_e - \sum_{\substack{e \in \widetilde{E} \\ x = e^-}} w_e \right) = 0. 
    \end{align*}
\end{proof}

\begin{remark}
    This last calculation is adapted from \cite{BakerFaber}. By applying the Abel-Jacobi theorem for metric graphs, we know that $\sum_{j = 1}^r m_j \int_{\gamma_j} \equiv 0$ as an element of $\Jac(\Gamma)$. This just means that this functional is identical to an integral along a closed path $\gamma \in H_1(\Gamma,\Z)$. However, in general
    $$
    \int_{X_v \cap \gamma} \neq 0 \in \Jac(X_v).
    $$
    Therefore, we cannot immediately conclude that
    $$
    \left[(0,\hdots,0),\sum_{j = 1}^r m_j \int_{\gamma_j}\right] = 0 \in \Jac(\X),
    $$
    and need to do this additional calculation.
\end{remark}

We now conclude with the proof of  Theorem \ref{thm:abel_jacobi_for_mCRS} (and thus Theorem \ref{mainthm_AbelJacobimCRS} from the introduction).

\begin{proof}[Proof of Theorem \ref{thm:abel_jacobi_for_mCRS}]
    By Proposition \ref{prop:abel_jacobi_maps_commute}, we have a commutative diagram
    \begin{equation}\label{eq_twoseslocal}\begin{tikzcd}
    	0 & {\bigoplus_{v \in V}\mathrm{Div_0}(X_v)} & {\mathrm{Div_0(\mathfrak{X})}} & {\mathrm{Div_0(\Gamma)}} & 0 \\
    	0 & {\bigoplus_{v \in V}\Jac(X_v)} & {\Jac(\mathfrak{X})} & {\Jac(\Gamma)} & 0
    	\arrow[from=1-1, to=1-2]
    	\arrow[from=1-2, to=1-3]
    	\arrow["{D \mapsto D_\Gamma}", from=1-3, to=1-4]
    	\arrow[from=1-4, to=1-5]
    	\arrow[from=2-1, to=2-2]
    	\arrow[from=2-2, to=2-3]
    	\arrow[from=2-3, to=2-4]
    	\arrow[from=2-4, to=2-5]
    	\arrow["A_{X,0}=\bigoplus_{v \in V} A_{v,0}",from=1-2, to=2-2]
    	\arrow["A_{\mathfrak{X},0}", from=1-3, to=2-3]
    	\arrow["A_{\Gamma,0}",from=1-4, to=2-4].
    \end{tikzcd}\end{equation}
    By the Abel-Jacobi-Theorem for compact Riemann surfaces, we already know that the kernel of the left vertical arrow $A_{X,0}=\bigoplus_{v \in V} A_{v,0}$ is given by $\bigoplus_{v\in V}\PDiv(X_v)$ and by the Abel-Jacobi-Theorem for metric graphs the kernel of the right vertical arrow $A_{\Gamma,0}$ is given by $\PDiv(\Gamma)$. 

    We now show that the kernel of the vertical arrow in the middle is given by $\PDiv(\mathfrak{X})$, i.e. that $D \in \Div_0(\X)$ is a principal divisor if and only if $A_{\X,0}(D) = 0$ in $\Jac(\X)$.
    
    Let $D = \div(\mathfrak{f}) \in \mathrm{PDiv}(\X)$ be a principal divisor. The function $\mathfrak{f}$ is given by a rational function $f_\Gamma \in \Rat(\Gamma)$ and meromorphic functions $f_v \in \mathcal{M}(X_v)$. Trivially extend $f_\Gamma$ to a function $F_\Gamma$ on $\X$ as above. We can now write $D$ as
    $$
    D = \sum_{v \in V} \div(f_v) + \div(F_\Gamma).
    $$
    Proposition \ref{prop:abel_jacobi_maps_commute} and the Abel-Jacobi theorem for Riemann surfaces imply
    $$
    A_{\X,0}\left( \sum_{v \in V} \div(f_v) \right) = 0,
    $$
    and Proposition \ref{prop:extended_rational_function_maps_to_zero} implies
    $$
    A_{\X,0}\big(\div(F_\Gamma)\big) = 0.
    $$
    This shows the \textit{only-if}-part.
    
    For the \textit{if}-part, suppose we are given a divisor $D \in \Div_0(\X)$ such that
    $$
    A_{\X,0}(D) = 0.
    $$
    By Proposition \ref{prop:abel_jacobi_maps_commute}, it follows that also $A_{\Gamma,0}(D_\Gamma) = 0$, and the Abel-Jacobi theorem for metric graphs, Theorem \ref{thm:abel_jacobi_theorem_for_metric_graphs}, implies the existence of a rational function $f_\Gamma \in \Rat(\Gamma)$ such that $\div(f_\Gamma) = D_\Gamma$. Choosing constant functions for each Riemann surface $X_v$, we can extend this function to a rational function $F_\Gamma \in \Rat(\X)$ on $\X$ satisfying $\div(F_\Gamma)_\Gamma = \div(f_\Gamma) = D_\Gamma$. By Proposition \ref{prop:extended_rational_function_maps_to_zero}, we still have
    $$
    A_{\X,0}\big(D - \div(F_\Gamma)\big) = 0.
    $$
    The exactness of the first and second row of Diagram \eqref{eq_twoses} in Proposition \ref{prop:abel_jacobi_maps_commute} also implies 
    $$
    D - \div(F_\Gamma) \in \bigoplus_{v \in V} \Div_0(X_v),
    $$
    and
    $$
    A_{X,0}\big(D - \div(F_\Gamma)\big) = 0 \in \bigoplus_{v \in V} \Jac(X_v).
    $$
    The Abel-Jacobi theorem for Riemann surfaces provides us with meromorphic functions $f_v \in \mathcal{M}(X_v)$ such that
    \begin{equation*}
        D - \div(F_\Gamma) = \sum_{v \in V} \div(f_v).  
    \end{equation*}
    Setting $\mathfrak{f} := \big(\{f_v\}_{v \in V}, f_\Gamma\big)$, we have
    $D = \div(\mathfrak{f})$
    since obviously
    $$
    \div(\mathfrak{f}) - \sum_{v \in V} \div(f_v) = \div(F_\Gamma).
    $$
    This shows that $D$ is principal.

    We now apply the snake lemma to Diagram \eqref{eq_twoseslocal} and, recalling that $\Div_0(X_v) \to \Jac(X_v)$ and $\Div_0(\Gamma) \to \Jac(\Gamma)$ are both surjective, we obtain an isomorphism of short exact sequences
    \[\begin{tikzcd}
    	0 & {\bigoplus_{v \in V}\mathrm{Pic}^0(X_v)} & {\mathrm{Pic}^0(\X)} & {\mathrm{Pic}^0(\Gamma)} & 0 \\
    	0 & {\bigoplus_{v \in V} \Jac(X_v)} & {\Jac(\X)} & {\Jac(\Gamma)} & 0
    	\arrow[from=1-1, to=1-2]
    	\arrow[from=1-2, to=1-3]
    	\arrow["A_{X,0} = \bigoplus_{v \in V} A_{v,0}",from=1-2, to=2-2]
    	\arrow["\small{{[D] \mapsto [D_\Gamma]}}"',from=1-3, to=1-4]
    	\arrow["{A_{\X,0}}", from=1-3, to=2-3]
    	%\arrow["u"', curve={height=12pt}, dashed, from=1-4, to=1-3]
    	\arrow[from=1-4, to=1-5]
    	\arrow["{A_{\Gamma,0}}", from=1-4, to=2-4]
    	\arrow[from=2-1, to=2-2]
    	\arrow[from=2-2, to=2-3]
    	\arrow[from=2-3, to=2-4]
    	\arrow[from=2-4, to=2-5]
    \end{tikzcd}\]
    as claimed.

    Finally, we construct a map $v : \Jac(\Gamma) \rightarrow \Jac(\X)$ such that $u = A_{\X,0}^{-1} \circ v \circ A_{\Gamma,0}$. Pick $v_0 \in V$. Let $\sum_j \int_{\gamma_j} \in \Jac(\Gamma)$ where $\gamma_j$ is a path in $\Gamma$ from $v_0$ to some point $a^j$ (since $A_{\Gamma,0}$ is surjective, we know that any element of $\Jac(\Gamma)$ can be represented in this way). As in Proposition \ref{prop_seshomology}, we use our set of basepoints $\mathcal{B}$ to once again obtain an embedding $\Gamma \hookrightarrow |\X|$. This allows us to extend the chain $\sum_j \gamma_j$ to a chain $\sum_j\delta_j$ on $|\X|$ and define
    $$
    v\left(\sum_j \int_{\gamma_j} \right) = \sum_j\left[\left(\int_{\delta_j \cap X_{v_1}},\hdots, \int_{\delta_j \cap X_{v_n}}\right),\left( \int_{j_*(\delta_j)}\right)\right].
    $$
    Note that $j_*(\delta_j)$ is just $\gamma_j$. To show that this map is well-defined (up to the choice of an embedding $\Gamma \hookrightarrow |\X|$), we need to show that an element of the form
    $$
    A_{\Gamma,0}(\div(f))
    $$
    goes to zero under $v$. If we write 
    $$
    A_{\Gamma,0}(\div(f)) = \sum_j \int_{\gamma_j},
    $$
    then, since
    \[\begin{tikzcd}
    	{Z_1(\Gamma)} & {Z_1(\X)} \\
    	{C_1(\Gamma)} & {C_1(\X)}
    	\arrow[from=1-1, to=1-2]
    	\arrow[hook, from=1-1, to=2-1]
    	\arrow[from=1-2, to=2-2]
    	\arrow[hook, from=1-2, to=2-2]
    	\arrow[from=2-1, to=2-2]
    \end{tikzcd},\]
    the closed chain $\sum_j \gamma_j$ will be sent to a closed chain $\sum_j \delta_j \in Z_1(\X)$ under the embedding $\Gamma \hookrightarrow |\X|$. Integrating along this chain represents $0$ in $\Jac(\X)$. Finally, it is clear from the construction that
    $$
    u\Big(\sum_j a^j\Big) = A_{\X,0}^{-1} \circ v \circ A_{\Gamma,0}\Big(\sum_j a^j\Big).
    $$
    The reason $v$ does not depend on the embedding $\Gamma \hookrightarrow |\X|$ itself, but only on the set of basepoints $\mathcal{B}$ is that choosing two different $1$-cells connecting $p_v$ with the marked points of $X_v$ leads to functionals differing by an integral along a closed path on $X_v$ which are also $0$ as elements of $\Jac(\X)$.
\end{proof}

%%%%%%%%%%%%%%%%%%%%%%%%%%%%%%%%%%%%%%%%%%%%%%%%%%%%%%

%%%%%%%%%%%%%%%%%%%%%%%%%%%%%%%%%%%%%%%%%%%%%%%%%%%%%%

\bibliographystyle{amsalpha}
\bibliography{biblio}{}

\end{document}